\documentclass[11pt,twoside, leqno]{article}

\usepackage{mathrsfs}
\usepackage{amssymb}
\usepackage{amsmath}
\usepackage{amsthm}
\usepackage{amsfonts}

\usepackage{color}

\allowdisplaybreaks

\def\rr{{\mathbb R}}
\def\rn{{{\rr}^n}}
\def\zz{{\mathbb Z}}

\def\nn{{\mathbb N}}
\def\fz{\infty}

\def\ccc{{\mathbb C}}

\def\cs{{\mathcal S}}

\def\az{\alpha}

\renewcommand\tilde{\widetilde}

\def\supp{{\rm{\,supp\,}}}

\def\ls{\lesssim}
\def\lz{\lambda}

\def\r{\right}
\def\lf{\left}

\def\bint{{\ifinner\rlap{\bf\kern.30em--}
\int\else\rlap{\bf\kern.35em--}\int\fi}\ignorespaces}

\def\sbint{{\ifinner\rlap{\bf\kern.32em--}
\hspace{0.078cm}\int\else\rlap{\bf\kern.45em--}\int\fi}\ignorespaces}

\def\dsup{\displaystyle\sup}

\newtheorem{theorem}{Theorem}[section]
\newtheorem{lemma}[theorem]{Lemma}

\newtheorem{proposition}[theorem]{Proposition}

\theoremstyle{definition}
\newtheorem{remark}[theorem]{Remark}
\newtheorem{definition}[theorem]{Definition}

\numberwithin{equation}{section}

\numberwithin{equation}{section}

\numberwithin{equation}{section}

\begin{document}

\arraycolsep=1pt

\title{\Large\bf   Herz-Type Hardy Spaces Associated with Ball Quasi-Banach Function Spaces
\footnotetext{\hspace{-0.35cm} {\it 2010 Mathematics Subject Classification}.
{Primary 42B30; Secondary 42B35, 46E30.}
\endgraf{\it Key words and phrases:} Herz-type Hardy space, ball quasi-Banach function space,
atom  decomposition, sublinear operator
\endgraf This project is supported by Natural Science Fundation of Xinjiang Province and NSFC (Nos. 2024D01C40, 12261083, 12161083).
 \endgraf $^\ast$\,Corresponding author
}}
\author{Aiting Wang, Wenhua Wang, Mingquan Wei and Baode Li$^\ast$}
\date{  }
\maketitle

\vspace{-0.8cm}

\begin{center}
\begin{minipage}{13cm}\small
{\noindent{\bf Abstract} \
Let $X$ be a ball quasi-Banach function space, $\alpha\in \mathbb{R}$ and $q\in(0,\infty)$. In this paper, the authors first introduce the Herz-type Hardy space $\mathcal{H\dot{K}}_{X}^{\alpha,\,q}({\mathbb {R}}^n)$, which is defined via the non-tangential grand maximal function. Under some mild assumptions on $X$, the authors establish the atomic decompositions of $\mathcal{H\dot{K}}_{X}^{\alpha,\,q}({\mathbb {R}}^n)$. As an application, the authors
 obtain
the boundedness of certain sublinear operators from $\mathcal{H\dot{K}}_{X}^{\alpha,\,q}({\mathbb {R}}^n)$ to $\mathcal{\dot{K}}_{X}^{\alpha,\,q}({\mathbb {R}}^n)$, where $\mathcal{\dot{K}}_{X}^{\alpha,\,q}({\mathbb {R}}^n)$ denotes the Herz-type space associated with ball
quasi-Banach function space $X$. Finally, the authors apply these results to three concrete function spaces: Herz-type Hardy spaces with variable
exponent, mixed Herz-Hardy spaces and Orlicz-Herz Hardy spaces, which belong to the family of Herz-type Hardy spaces associated with ball quasi-Banach function spaces.
}
\end{minipage}
\end{center}

\section{Introduction}

\quad The theory of Herz spaces, pioneered by Beurling and Herz, is of importance in the study of convolution algebra and Fourier multipliers (refer to \cite{b85, b64,h68}). The exploration of Herz-type Hardy spaces, which originated in the late 1980s to provide a natural analog to the classical Hardy space, has become a highly captivating subject of research in Harmonic Analysis.
 In 1989, Chen et al. \cite{c89} and Garc\'{i}a-Cuerva \cite{cc89} introduced what are now known as non-homogeneous Herz-type Hardy spaces in $n=1$  (see \cite{c89}) and $n\geq2$ (see \cite{cc89}), and subsequently established their central atomic decomposition.  After that, Lu et al. delved into homogeneous Herz-type Hardy spaces and conducted a systematic study of Herz-type Hardy spaces with general indices.  Their work showed that Herz-type Hardy spaces are an appropriate localized version of the classical Hardy spaces and possess several real-variable properties, such as maximal function characterizations, atomic and molecular decompositions (see monograph \cite{l08} for details). Thanks to the groundbreaking work of Chen, Garc\'{i}a-Cuerva and Lu et al., the extension of Herz-type Hardy spaces has been extensively studied, and their real-variable theory has been significantly developed. Examples include weighted Herz-type Hardy spaces \cite{l97,q20}, Herz-type Hardy spaces with variable exponents \cite{wl12}, and  Herz-Morrey-Hardy spaces with variable exponents \cite{xx15}. Most recently, Zhao et al. \cite{z22}  studied mixed Herz-Hardy spaces.


On the other hand,  in 2017, Sawano et al. \cite{s17} first introduced the ball quasi-Banach
function space $X$, which generalized  the Banach function space in \cite{b88}. The ball quasi-Banach function spaces encompass many well-known function spaces, including
weighted Lebesgue spaces, Morrey spaces, mixed-norm Lebesgue spaces and  Orlicz spaces, etc (see \cite[Section 7]{s17} and \cite[Section 5]{wy19}).  It's noteworthy that not all of these spaces qualify as quasi-Banach function spaces (for further details, refer to \cite{s17,z20}). Therefore, this peculiarity actually broadens the scope of applications for studying , as it accommodates more diverse functional behaviors. Moreover, Sawano et al. \cite{s17} not only introduced but also thoroughly explored the Hardy space $H_X(\rn)$
associated with $X$.  Under the assumptions that the Fefferman-Stein vector-valued maximal inequality holds on $X$ and the powered Hardy-Littlewood maximal operator is bounded on the associated space of $X$ , they obtained the atom decompositions of $H_X(\rn)$.
Indeed, the theory of the Hardy space
$H_X(\rn)$ provides an remarkably suitable framework. This framework serves as a unifying force,  bringing together the theories of numerous variants of Hardy spaces that had been previously explored. Classical Hardy spaces, mixed-norm Hardy spaces, variable Hardy spaces, and Orlicz Hardy spaces all find their place within this  construct (see \cite[Chapter 7]{s17} for the details). Subsequently, Zhang et al. \cite{z20} and
Wang et al. \cite{wy19} introduced the weak Hardy space $WH_X(\rn)$ associated with $X$ and developed
its complete real-variable theory.  Very recently, Wei et al. \cite{w22} considered  Herz-type space associated with  $X$. For more recent developments about ball quasi-Banach function spaces, we refer the reader to \cite{c20,is17,t21,yh22}.

Inspired by the groundbreaking work of Herz et al. \cite{h68}, Lu et al. \cite{l08} and  Sawano et al. \cite{s17}, it is natural and interesting to
ask whether
 there exist a unifying framework for Herz-type Hardy spaces  which encompasses the classical Herz-type Hardy spaces, the Herz-type Hardy spaces with variable
exponent and the mixed Herz-type Hardy spaces as special cases. In this paper, we shall answer this
problem affirmatively  by introducing  new Herz-type Hardy space $\mathcal{H\dot{K}}_{X}^{\alpha,\,q}({\mathbb {R}}^n)$ associated with ball
quasi-Banach function space $X$.

To be precise, this article is organized as follows.

In Section 2, we recall some notions concerning the ball (quasi)-Banach function space $X$, and introduce the definitions of Herz-type Hardy spaces $\mathcal{H}\dot{\mathcal{K}}_{X}^{\alpha,\,q}(\rn)$ associated with the ball quasi-Banach function spaces $X$.  Section 3 is devoted to establishing the atomic decomposition of  $\mathcal{H}\dot{\mathcal{K}}_{X}^{\alpha,\,q}(\rn)$ (see Theorem \ref{t3.1} below). As an application of the atomic characterization of  $\mathcal{H}\dot{\mathcal{K}}_{X}^{\alpha,\,q}(\rn)$, we obtain the boundedness of certain sublinear operators from $\mathcal{H}\dot{\mathcal{K}}_{X}^{\alpha,\,q}(\rn)$ to $\dot{\mathcal{K}}_{X}^{\alpha,\,q}(\rn)$ (see Theorem \ref{t5.1} below). In Section 4, we apply the results obtained in this article to three specific examples: Herz-type Hardy spaces with variable exponents (detailed in Subsection \ref{s4.1}), mixed Herz-type Hardy spaces (detailed in Subsection \ref{s4.2}) and  Herz-Orlicz Hardy spaces (detailed in Subsection \ref{s4.3}). It is particularly worthy of note that the results derived in this article possess  a remarkable degree of generality. To be more specific, the atomic decompositions of $\mathcal{H}\dot{\mathcal{K}}_{X}^{\alpha,\,q}(\mathbb{R}^n)$ (see Theorem \ref{t3.1} below) and the boundedness of certain sublinear operators from $\mathcal{H}\dot{\mathcal{K}}_{X}^{\alpha,\,q}(\mathbb{R}^n)$ to $\dot{\mathcal{K}}_{X}^{\alpha,\,q}(\mathbb{R}^n)$ (see Theorem \ref{t5.1} below) encompass the corresponding results of the Herz-type Hardy spaces with variable exponent (see \cite[Theorems 2.1 and 2.2]{wl12}) and the mixed Herz-Hardy spaces (see \cite[Theorems 3.1 and 5.1]{z22}) as special cases. Furthermore, our work seems pioneering in the investigation of Herz-Orlicz Hardy spaces, as this particular class of spaces has not been previously explored in the literature.

Finally, we make some conventions on notation. Let $\nn:=\{0,1,\, 2,\,\ldots\}$ and  $\zz_+:=\{1,\, 2,\,\ldots\}$.
For any $\az:=(\az_1,\ldots,\az_n)\in\nn_0^n$,
$|\az|:=\az_1+\cdots+\az_n$,
and $\partial^\az:=
(\frac{\partial}{\partial x_1})^{\az_1}\cdots
(\frac{\partial}{\partial x_n})^{\az_n}$.
Throughout the whole paper, we denote by $C$ a positive
constant which is independent of the main parameters, but it may vary from line to line.
The symbol $D\ls F$ means that $D\le CF$. If $D\ls F$ and $F\ls D$, we then write $D\sim F$.
For any set $E \subset \rn$, we use $E^\complement$ to denote the set $\rn\setminus E$.
Let $\cs(\rn)$ be the space of Schwartz functions, $\cs'(\rn)$ the space of tempered distributions,
and $C^N(\rn)$ the space of  continuously differentiable functions of order $N$.



\section{Preliminaries}
\hskip\parindent
In this section, we recall the definition of the ball quasi-Banach
function space $X $ introduced in \cite{s17}. Following this, we introduce the definitions of the  Herz-type Hardy spaces $\mathcal{H}\dot{\mathcal{K}}_{X}^{\alpha,\,q}(\rn)$  associated with the ball quasi-Banach function spaces $X$.

 We always use the symbol $\mathfrak{M}(\rn)$ to
denote the set of all measurable functions on $\rn$. For any
$x\in\rn$ and $r\in(0,\,\infty)$, let $B(x,\,r):=\{y\in \rn :|x-y|<r\}$ and
\begin{align*}
\mathfrak{{B}}:=\{B(x,\,r): x\in\rn \ \ {\text{and}}\  \ r\in(0,\,\infty)\}.
\end{align*}

 \begin{definition}\label{d2.1xx}
 A quasi-Banach space $X\subset\mathfrak{M}(\rn)$ is called a {\it ball quasi-Banach function space} if it
satisfies
\begin{enumerate}
\item[\rm{(i)}]
$\|f\|_X=0$ implies that $f=0$ almost everywhere;
\item[\rm{(ii)}] $|g|\leq|f|$ almost everywhere implies that $\|g\|_X\leq\|f\|_X$;
\item[\rm{(iii)}] $0\leq f_m\uparrow f$ as $m\rightarrow\infty$ almost everywhere implies that $\|f_m\|_X\uparrow\|f\|_X$ as $m\rightarrow\infty$;
\item[\rm{(iv)}] $B\in\mathfrak{B}$ implies that $\chi_B\in X$.
\end{enumerate}
 \end{definition}

Furthermore, a ball quasi-Banach function space $X$ is called a ball Banach function
space if the norm of $X$ satisfies that for any $f,\,g\in X$,
\begin{align*}
\|f+g\|_{X}\leq\|f\|_{X}+\|g\|_{X}
\end{align*}
and, for any $B\in\mathfrak{B}$, there exists a positive constant $C_{(B)}$, depending on $B$, such that,
for any $f\in X$,
\begin{align*}
\int_{B}|f(x)|dx \leq C_{(B)}\|f\|_X.
\end{align*}

The associate space $X'$ of any given ball Banach function space $X$ is defined as follows ( see \cite[
Chapter 1, Section 2]{b88} or \cite{s17} for more details).
\begin{definition}
 For any given ball Banach function space $X$, its {\it associate space (also called the K\"{o}the
dual space)} $X'$ is defined by setting
$$X':=\lf\{f\in\mathfrak{M}(\rn):\|f\|_{X'}<\infty\r\},$$
where, for any $f\in\mathfrak{M}(\rn)$,
$\|f\|_{X'}:=\sup\lf\{\|fg\|_{L^1}:g\in X, \|g\|_X=1\r\}$
and $\|\cdot\|_{X'}$ is called the {\it associate norm} of $\|\cdot\|_{X}$.
 \end{definition}

 \begin{remark}\label{r2.3a}
 \begin{enumerate}
\item[\rm{(i)}]
By \cite[Proposition 2.3]{s17}, we obtain that, if $X$ is a ball Banach function
space, then its associate space $X'$ is also a ball Banach function space.
\item[\rm{(ii)}]  As was mentioned in \cite[Section 7]{s17} and \cite[Section 5]{wy19}, the family of ball quasi-Banach function spaces includes  Morrey spaces \(M_{r}^{p}(\mathbb{R}^{n})\)(see \cite{m38}),  mixed-norm Lebe-\\sgue spaces $L^{\vec{p}}(\rn)$ (see \cite{b61}), variable Lebesgue spaces $L^{p(\cdot)}(\mathbb{R}^{n})$ (see \cite{dhh11}), weighted Lebesgue spaces $L^{p}_\omega(\rn)$ (see \cite{m72}) and Orlicz spaces \(L^{\Phi}(\mathbb{R}^n)\) (see \cite{r91}), which are not necessary to be Banach
function spaces.
\end{enumerate}
\end{remark}

\begin{lemma}\rm{\cite[ Lemma 2.6]{z20}}\label{l20}
 Every ball Banach function space $X$ coincides with its second associate
space $X''$. In other words, a function $f$ belongs to $X$ if and only if it belongs to $X''$
and, in that case,
$$\|f\|_{X}=\|f\|_{X''}.$$
\end{lemma}

Throughout this paper, always let $B_k:=\{x\in\rn:|x|\leq2^k\}$ and $D_k:=B_k\setminus B_{k-1}$ for $k\in\mathbb{Z}$. Denote briefly the characteristic function $\chi_k$: =$\chi_{D_k}$ for any $k\in\mathbb{Z}$. For any quasi-Banach space $X$, the space $L^X
_{\text{loc}}(\rn)$ consists of all functions $f\in\mathfrak{M}(\rn)$
such that $f\chi_{F}\in X$ for any compact subset $F\subset \rn$. Now we recall  the definition of  the {{ homogeneous Herz-type space}} $\dot{\mathcal{K}}_{X}^{\alpha,\,q}(\rn)$.

\begin{definition}\rm{\cite[Definition 2.5]{w22}}
Let $X$ be a ball quasi-Banach function space,  $\alpha\in \mathbb{R}$ and $q\in(0,\infty]$.
The {\it{homogeneous Herz-type space}}  $\dot{\mathcal{K}}_{X}^{\alpha,\,q}(\rn)$  is defined by setting,
\begin{eqnarray*}
\dot{\mathcal{K}}_{X}^{\alpha,\,q}(\rn):=\lf\{f\in L^X
_{\text{loc}}(\rn\setminus\{0\}):\|f\|_{\dot{\mathcal{K}}_{X}^{\alpha,\,q}(\rn)}<\infty\r\},
\end{eqnarray*}
where
\begin{eqnarray*}
\|f\|_{\dot{\mathcal{K}}_{X}^{\alpha,\,q}(\rn)}:=\lf\{\sum_{k\in\zz}2^{k\alpha q}\lf\|f\chi_k\r\|
_{X}^q\r\}^{1/q}.
\end{eqnarray*}

Here, there is the usual modification when $q=\infty$.
\end{definition}

 \begin{remark}\label{r2.6a}
When \(X = L^{p}(\mathbb{R}^{n})\) with \(p \in (0, \infty)\), the space \(\dot{\mathcal{K}}_{X}^{\alpha, q}(\mathbb{R}^{n})\)  is reduced to the classical homogeneous Herz space \(\dot{\mathcal{K}}_{{p}}^{\alpha, q}(\mathbb{R}^{n})\) (see \cite[Definition 1.1.1]{l08}). In this specific case, we have \(\dot{K}_{q}^{0, q}(\mathbb{R}^{n}) = L^{q}(\mathbb{R}^{n})\) and \(\dot{K}_{q}^{\alpha, q}(\mathbb{R}^{n}) = L^{q}(\mathbb{R}^{n}, |x|^{\alpha q})\) for any \(\alpha \in \mathbb{R}\), where \(L^{q}(\mathbb{R}^{n}, |x|^{\alpha q})\) is the Lebesgue space \(L^{q}(\mathbb{R}^{n})\) ) with power weight \(|x|^{\alpha q}\). Consequently, the Herz spaces can be regarded as a natural generalization of the Lebesgue spaces with power weights.
\end{remark}

In what follows,  we introduce the definitions of the homogeneous Herz-type Hardy spaces $\mathcal{H}\dot{\mathcal{K}}_{X}^{\alpha,\,q}(\rn)$ associated with ball quasi-Banach function spaces $X$ (shortly  Herz-type Hardy space associated with ball quasi-Banach function spaces $X$),  via the non-tangential grand maximal function $M_N(f)$.

A $C^\infty$ function $\varphi$ is said to belong to the Schwartz class $\cs$, if
for all multi-indices $\alpha$ and $\beta$, we have
$\|\varphi\|_{\alpha,\,\beta}:=\dsup_{x\in\rn}|x^\alpha\partial^\beta\varphi(x)|<\infty$.
The dual space of $\cs$, namely, the space of all tempered distributions on $\rn$ equipped with the weak-$\ast$
topology, is denoted by $\cs'$. For any $f\in\cs'$, let $M_Nf$ be the non-tangential grand maximal function of $f$ defined by
\begin{align*}
M_{N}(f)(x):=\sup_{\varphi\in \cs_N}M_{\varphi}(f)(x),
\end{align*}
where $\cs_N:=\{\varphi\in\cs:\ \dsup_{|\alpha|,\,|\beta|\leq N}\|\varphi\|_{\alpha,\,\beta}\leq1\} $ with $N>n+1$, and $M_{\varphi}(f)$ is the non-tangential maximal operator $M_{\varphi}(f)$ defined by
\begin{align*}
M_{\varphi}(f)(x):=\sup_{|y-x|<t}|f*\varphi_{t}(y)|
\end{align*}
with $\varphi_{t}(x):=t^{-n}\varphi(x/t)$.

\begin{definition}\label{da2.7}
 Let $X$ be a ball quasi-Banach function space,  $\alpha\in \mathbb{R}$, $q\in(0,\infty)$ and $N>n+1$.
The {\it{Herz-type Hardy spaces}} $\mathcal{H}\dot{\mathcal{K}}_{X}^{\alpha,\,q}(\rn)$  associated with ball quasi-Banach function spaces $X$  are defined  by setting,
\begin{eqnarray*}
\mathcal{H}\dot{\mathcal{K}}_{X}^{\alpha,\,q}(\rn):=\lf\{f\in \mathcal{S}^{'} :M_N(f)\in \dot{\mathcal{K}}_{X}^{\alpha,\,q}(\rn)\r\},
\end{eqnarray*}
where
$$
\|f\|_{\mathcal{H}\dot{\mathcal{K}}_{X}^{\alpha,\,q}(\rn)}=\lf\|M_N(f)\r\|_{\dot{\mathcal{K}}_{X}^{\alpha,\,q}(\rn)}.$$
\end{definition}

 \begin{remark}\label{r2.6b}
  It is known from Remark \ref{r2.3a} that Morrey spaces \(M_{r}^{p}(\mathbb{R}^{n})\),  mixed-norm Lebesgue spaces $L^{\vec{p}}(\rn)$, variable Lebesgue spaces $L^{p(\cdot)}(\mathbb{R}^{n})$, weighted Lebesgue spaces $L^{p}_\omega(\rn)$ and Orlicz-slice spaces \((E_{\Phi}^{r})_t(\mathbb{R}^n)\) are all members of the ball quasi-Banach function space $X$. For \(X\)  equal to any of these spaces, the Herz-type Hardy space $\mathcal{H}\dot{\mathcal{K}}_{X}^{\alpha,\,q}(\rn)$ becomes the corresponding variant of  Herz-type Hardy space, respectively. A particularly noteworthy observation is that when $X=L^{p(\cdot)}(\rn)$, $\mathcal{H}\dot{\mathcal{K}}_{X}^{\alpha,\,q}(\rn)$ is reduced to the Herz-type Hardy spaces with variable
exponent $\mathcal{H}\dot{\mathcal{K}}_{L^{p(\cdot)}}^{\alpha,\,q}(\rn)$ studied in \cite{dhh11}.  Similarly, when  $X=L^{\vec{p}}(\rn)$, $\mathcal{H}\dot{\mathcal{K}}_{X}^{\alpha,\,q}(\rn)$ is reduced to the mixed Herz-Hardy spaces $\mathcal{H}\dot{\mathcal{K}}_{L^{\vec{p}}}^{\alpha,\,q}(\rn)$  studied in \cite{z22}. Furthermore, when \(X = L^{\Phi}(\mathbb{R}^n)\), $\mathcal{H}\dot{\mathcal{K}}_{X}^{\alpha,\,q}(\rn)$  give rise to the  Orlicz-Herz Hardy space  $\mathcal{H}\dot{\mathcal{K}}_{L^{\Phi}}^{\alpha,\,q}(\rn)$, which doesn't seem to have been explored before. This demonstrates the wide-ranging generality of the research on Herz- type Hardy spaces $\mathcal{H}\dot{\mathcal{K}}_{X}^{\alpha,\,q}(\rn)$.
\end{remark}
Since there is no concrete norm expression for the ball Banach function space $X$, we need a weak assumption about the boundedness of  Hardy-Littlewood maximal operator on $X$ and its associated space $X'$.

For any $f\in L_{\mathrm{loc}}^1(\rn)$
(the set
of all locally integrable functions on $\rn$), {\it Hardy-Littlewood maximal function} $M_{\mathrm{HL}}(f)$ is defined by
\begin{align*}
M_{\mathrm{HL}}(f)(x):=\sup_{r>0}\frac{1}{|B(x,\,r)|}\int_{B(x,\,r)}|f(y)|\,dy,
\end{align*}
where $B(x,\,r)\in\mathfrak{{B}}$.\\\\
\textbf{ Assumption\,1.\,} Let $X$ be a ball Banach function space with and let $X'$ be the associate space of $X$. Then there exist  positive constants $C_1$ and $C_2$ such that for any $f\in X$,
\begin{align}\label{e3.2}
\|M_{\mathrm{HL}}(f)\|_X\leq C_1\|f\|_X
\end{align}
and
\begin{align*}
 \|M_{\mathrm{HL}}(f)\|_{X'}\leq C_2\|f\|_{X'}.
\end{align*}

\begin{lemma}\label{l3.1x}
Let $X$ be a ball Banach function space and satisfy  Assumption\,1. Then there exist  a positive constant $C$ and $0<\delta_{X}<1$
 such that for all balls $B\in\mathfrak{B}$ and all measurable sets $E\subset B$ ,
 \begin{enumerate}
\item[\rm{(i)}]$\frac{\|\chi_B\|_{X}}{\|\chi_E\|_{X}}\leq C\lf(\frac{|B|}{|E|}\r);$
\item[\rm{(ii)}]$\frac{\|\chi_E\|_{X}}{\|\chi_B\|_{X}}\leq C\lf(\frac{|E|}{|B|}\r)^{\delta_X}.$
\end{enumerate}
\end{lemma}
 The proof of Lemma \ref{l3.1x} (i)  is similar to
that of  \cite[Lemma 5]{i17}; we omit the details. Lemma \ref{l3.1x} (ii) is from \cite[Lemma 2.4]{w22}.
\begin{remark}\label{rc2.10}
From Remark \ref{r2.3a} (i) and Lemma \ref{l20}, it can be seen that the result of Theorem 2.8 still holds for the  associate space $X'$ of $X$. In this case, we denote the constant $\delta_{X}$ as $\delta_{X'}$.
\end{remark}
Throughout this paper, $\delta_{X}$ and $\delta_{X'}$ are the same as in Lemma \ref{l3.1x} and Remark \ref{rc2.10}, respectively.
The  boundedness of $M_{\mathrm{HL}}$
 on $\dot{\mathcal{K}}_{X}^{\alpha,\,q}(\rn)$ is given in \cite[Theorem 3.3]{w22} as follows.
\begin{lemma}\label{l3.4x'}
Let  $X$ be a ball Banach function space and satisfy  \eqref{e3.2} and  $ X' = X^*$, where $X^*$ denote the dual space of $X$, and let  $-n\delta_X<\alpha<n\delta_{X'}$ and $q\in(0,\infty)$. Then Hardy-Littlewood maximal function $M_{\mathrm{HL}}$ is bounded on $\dot{\mathcal{K}}_{X}^{\alpha,\,q}(\rn)$, respectively.
\end{lemma}
\begin{remark}\label{r2.10}
 It is known from \cite[Remark 3.1]{w22} that a ball Banach function space $X$ has an absolutely continuous norm if
and only if $ X' = X^*$. Thus, the condition $ X' = X^*$
is satisfied for a
large class of ball Banach function spaces.  Many function spaces, such as Lebesgue
spaces, variable exponent Lebesgue spaces, mixed-norm Lebesgue spaces, Lorentz
spaces and Orlicz spaces have absolutely continuous norms under some mild conditions, and hence, they satisfy $ X' = X^*$.
\end{remark}

\begin{proposition}{\label{p2.7}}
Let  $X$ be a ball quasi-Banach function space and satisfy  Assumption\,1 and  $ X' = X^*$, where $X^*$ denote the dual space of $X$, and let  $0<\alpha<n\delta_{X'}$ and $q\in(0,\infty)$. Then
\begin{align*}
\mathcal{H}\dot{\mathcal{K}}_{X}^{\alpha,\,q}(\rn)\cap L^X
_{\mathrm{loc}}(\rn\setminus\{0\}) =\dot{\mathcal{K}}_{X}^{\alpha,\,q}(\rn).
\end{align*}

\end{proposition}

\begin{proof}
Let $X$, $\alpha$ and $q$ be as in the present proposition. From \cite{s71}, it is easy to see that
\begin{align}\label{e2.4xx}
M_N(f)(x)\leq C M_{\rm HL}(f)(x),\ \ x\in\rn,
\end{align}
which, combined with Lemma \ref{l3.4x'}, further yield that  $M_N f$ is bounded on $\dot{\mathcal{K}}_{X}^{\alpha,\,q}(\rn)$ and ${\mathcal{K}}_{X}^{\alpha,\,q}(\rn)$, respectively. By this, we easily obtain that Proposition \ref{p2.7} holds true.
\end{proof}



\section{The atomic decompositions of $\mathcal{H}\dot{\mathcal{K}}_{X}^{\alpha,\,q}(\rn)$ and $\mathcal{H}{\mathcal{K}}_{X}^{\alpha,\,q}(\rn)$}
\hskip\parindent
In this section, we establish atomic decompositions of the Herz-type Hardy spaces $\mathcal{H}\mathcal{K}_{X}^{\alpha,\,q}(\rn)$  associated with ball quasi-Banach function spaces $X$.
We first begin with
the following notions of $(X,\,\alpha,\,s)$-atom.

\begin{definition}\label{d3.1}
Let $X$ be a ball Banach function space, $\alpha \in [n\delta_{X'},\,\infty)$ and  $s\in[(\alpha-n\delta_{X'}),\,\infty)\cap\zz_+$. A {\it central $(X,\,\alpha,\,s)$-atom} is a measurable function $a$ on $\rn$ satisfying
\begin{enumerate}
\item[\rm{(i)}] (support condition) $\supp a\subset B_r:=B(0,\,r)\in\mathfrak{B}$;
\item[\rm{(ii)}] (size condition) $\|a\|_{X}\le {|B(0,\,r)|^{-\alpha/n}}$;
\item[\rm{(iii)}]  (vanishing moment condition) $\int_\rn a(x)x^\beta dx=0$ for any $\beta\in \mathbb{Z}^n_+$ with $|\beta|\leq s$.
\end{enumerate}
\end{definition}

\begin{definition}\label{d3.3xx}
 Let $X$ be a ball Banach function space, $q\in(0,\,\infty)$, $\alpha \in [n\delta_{X'},\,\infty)$ and $s\in[(\alpha-n\delta_{X'}),\,\infty)\cap\zz_+$.
  The {\it  Herz-type  atom Hardy space}
$\mathcal{H}\dot{\mathcal{K}}^{\alpha,\,s,\,q}_{X,\,\rm{atom}}(\rn)$ associated with $X$ is defined to be the set of all tempered distribution
$f\in\mathcal{S'}$ of the form $f=\sum_{j\in\zz}\lambda_j a_j$ in $\cs'$,
where $\{a_j\}_{j\in\zz}$ are central $(X,\,\alpha,\,s)$-atoms as in Definition \ref{d3.1} respectively, supported on $\{B_{r_j}\}_{j\in\zz}\subset\mathfrak{B}$, $\{\lambda_j\}_{j\in\zz}\subset\ccc$
and $\sum_{j\in\zz}|\lambda_j|^q<\infty.$
Moreover, the quasi-norm of $f\in \mathcal{H}\dot{\mathcal{K}}^{\alpha,\,s,\,q}_{X,\,\rm{atom}}$ is defined by
$$\|f\|_{\mathcal{H}\dot{\mathcal{K}}^{\alpha,\,s,\,q}_{X,\,\rm{atom}}(\rn)}:=\inf \lf(\sum_{j\in\zz}|\lambda_j|^q\r)^{1/q},$$
where the infimum is taken over all above decompositions of $f$.
\end{definition}

Now we state the main result as follows.

\begin{theorem}\label{t3.1}
Let $X$ be a ball Banach function space and satisfy  Assumption\,1, and let $q\in(0,\,\infty)$, $\alpha \in [n\delta_{X'},\,\infty)$ and $s\in[(\alpha-n\delta_{X'}),\,\infty)\cap\zz_+$. Then we have
$$\mathcal{H}\dot{\mathcal{K}}_{X}^{\alpha,\,q}(\rn)
=\mathcal{H}\dot{\mathcal{K}}^{\alpha,\,s,\,q}_{X,\,\rm{atom}}(\rn)$$
 with equivalent quasi-norms.
\end{theorem}
To prove Theorem \ref{t3.1}, we need two crucial lemmas as follows.
\begin{lemma}\rm{\cite[Lemma 2.5]{z20}}\label{l3.4x}
Let $X$ be a ball Banach function space with associate space $X'$. If $f\in X$ and $g\in X'$, then $fg$ is integrable on $\rn$ and
$$\int_{\rn}|f(x)g(x)|\,dx\leq C\|f\|_{X}\|g\|_{X'}.$$
This inequality is named the generalized H\"{o}lder inequality.
\end{lemma}

\begin{lemma}\rm{\cite[Lemma 2.2]{is17}}\label{l3.5x}
 Let $X$ be a ball Banach function space. Suppose that the Hardy
 Littlewood maximal operator $M_{\mathrm{HL}}$ is weakly bounded on $X$, that is, there exists a positive
 constant $C'$ such that,
 \begin{equation}\label{ea3.3x}
 \|\chi_{\{M_{\mathrm{HL}}f>\lambda\}}\|_{X}\leq C' \lambda^{-1}\|f\|_{X}
  \end{equation}
  is true for all $f\in X$ and $\lambda >0$.
 Then there exists a positive constant $C>0$ such that for any $B \in {\mathfrak{B}}$,
$$\frac{1}{|B|}\|\chi_B\|_{X}\|\chi_B\|_{X'}\leq C.$$
\end{lemma}

\begin{proof}[Proof of Theorem \ref{t3.1}] The proof is divided into two steps.

\textbf{Step 1.} In this step, we show that
\begin{equation}\label{e3.3xx}
\mathcal{H}\dot{\mathcal{K}}^{\alpha,\,s,\,q}_{X,\,\rm{atom}}(\rn)
\subset\mathcal{H}\dot{\mathcal{K}}_{X}^{\alpha,\,q}(\rn).
\end{equation}
Let $f\in\mathcal{H}\dot{\mathcal{K}}^{\alpha,\,s,\,q}_{X,\,\rm{atom}}(\rn)$. By Definition \ref{d3.3xx},
we know that there exist $\{\lambda_{j}\}_{j\in\mathbb{Z}}\subset\mathbb{C}$ and
a sequence of central $(X,\,\alpha,\,s)$-atoms  $\{a_{j}\}_{j\in\mathbb{Z}}$,
supported, respectively, on $\{B_{r_j}\}_{j\in\zz}\subset\mathfrak{B}$
such that
$$
f=\sum_{j\in\mathbb{Z}}\lambda_j a_j     \quad {\rm in}\ \cs'
$$
and
$$\|f\|_{\mathcal{H}\dot{\mathcal{K}}^{\alpha,\,s,\,q}_{X,\,\rm{atom}}(\rn)}\sim \lf(\sum_{j\in\zz}|\lambda_j|^q\r)^{1/q}<\infty.$$
Here, without loss of generality, suppose that $B_{r_j}:=B_j=\{x:\,|x|\leq 2^{j}\}$ for $j\in\zz$. To prove $f\in\mathcal{H}\dot{\mathcal{K}}_{X}^{\alpha,\,q}(\rn)$, we consider two cases: $0<q\leq1$
and $1<q<\infty$.

\textbf{Case 1:} $0<q\leq1$. It suffices to prove that there exists a positive constant $C$ such that for each central
$(X,\,\alpha,\,s)$-atom $a$,
\begin{align*}
\|M_N(a)\|_{\mathcal{H}\dot{\mathcal{K}}_{X}^{\alpha,\,q}(\rn)}\leq C,
\end{align*}
where the constant $C$ is independent of $a$. For any given central
$(X,\,\alpha,\,s)$-atom $a$ with supp\,$a
\subset B_{k_0}=\{x:\,|x|\leq2^{k_0}\}$ for some $k_0\in\zz$, we have
\begin{align*}
\|M_N(a)\|_{\dot{\mathcal{K}}_{X}^{\alpha,\,q}(\rn)}^{q}&=\sum_{k=-\infty}^{k_0+3}|B_k|^{\alpha q/n}\lf\|
M_N(a)\chi_k\r\|_{X}^q\\
&\ \ \ +\sum_{k=k_0+4}^{\infty}|B_k|^{\alpha q/n}\lf\|
M_N(a)\chi_k\r\|_{X}^q\\
&=:{\rm I+II}.
\end{align*}
First, let us deal with I.  From \eqref{e3.2}, \eqref{e2.4xx} and the size condition of $a$, it follows that
\begin{align*}
{\rm I}&\leq\lf\|
M_N(a)\r\|_{X}^q\sum_{k=-\infty}^{k_0+3}|B_k|^{\alpha q/n}
\leq\lf\|
M_{\rm HL}(a)\r\|_{X}^q\sum_{k=-\infty}^{k_0+3}|B_k|^{\alpha q/n}
\leq\|
a\|_{X}^q\sum_{k=-\infty}^{k_0+3}|B_k|^{\alpha q/n}\leq C.
\end{align*}

For II, we need a pointwise estimate for $M_N(a)(x)$ on $D_k$. Let $\varphi\in \mathcal{S}_N$, $s\in \nn$
such that $\alpha-n\delta_{X'}<s+1$.  Denote by $P_s$ the $s-$th order Taylor series expansion of $\varphi$. If
$|x -y| <t$, then by the vanishing moment condition of $a$, Taylor Remainder Theorem and supp\,$a\subset B(0,\,2^{k_0})$, we obtain that
\begin{align*}
|a*\varphi_t(y)|&=t^{-n}\lf|\int_{\rn}a(z)\lf( \varphi \lf(\frac{y-z}{t}\r)-P_s\lf(-\frac{z}{t}\r)\r) dz\r|\\
&\leq C t^{-n}\int_{\rn}|a(z)|\lf|\frac{z}{t}\r|^{s+1}\lf(1+\frac{|y-\theta z|}{t}\r)^{-(n+s+1)} dz\\
 &\leq C \int_{B_{k_0}}|a(z)||{z}|^{s+1}(
 t+{|y-\theta z|})^{-(n+s+1)} dz,
\end{align*}
where $\theta\in (0,\,1)$. By $x\in D_k$ and $k\geq k_0+4$, we have $|x|>2^{k_0+2}$. From $|x -y| <t$ and $z\in B_{k_0}$, we conclude that
$$t+|y-\theta z|\geq |x-y|+|y-\theta z|\geq|x|-|z|\geq\frac{|x|}{2}.
$$
By this, Lemma \ref{l3.4x} and the size condition of $a$, we obtain that
\begin{align*}
|a*\varphi_t(y)|
 &\leq C\int_{B_{k_0}}|a(z)||{z}|^{s+1}(
 t+{|y-\theta z|})^{-(n+s+1)} dz\\
 &\leq C 2^{k_0(s+1)}|x|^{-(n+s+1)}\int_{B_{k_0}}|a(z)| dz\\
&\leq C 2^{k_0(s+1)}|x|^{-(n+s+1)}\|a\|_{X}\|\chi_{B_{k_0}}\|_{X'}\\
&\leq C 2^{k_0(s+1)}|x|^{-(n+s+1)}|B_{k_0}|^{-\alpha/n}\|\chi_{B_{k_0}}\|_{X'}.
\end{align*}
Thus, for any $x\in D_k$ and $k\geq k_0+4$,
\begin{align}\label{ea3.3xx}
M_N(a)(x)
\leq C 2^{k_0(s+1)-k(n+s+1)}|B_{k_0}|^{-\alpha/n}\|\chi_{B_{k_0}}\|_{X'}.
\end{align}
Notice that  \eqref{e3.2} implies \eqref{ea3.3x}. Indeed, suppose that  \eqref{e3.2} holds true. By Definition \ref{d2.1xx} (ii), we obtain that
$$ \|\chi_{\{M_{\mathrm{HL}}f>\lambda\}}\|_{X}\leq\lf \|\frac{M_{\mathrm{HL}}f}{\lambda}\r\|_{X} \leq C\lambda^{-1}\|f\|_{X}
  $$
Thus, when \eqref{ea3.3x} is replaced by  \eqref{e3.2} in Lemma \ref{l3.5x}, the Lemma \ref{l3.5x} still holds true. From \eqref{ea3.3xx}, \eqref{e3.2}, Lemma \ref{l3.5x} and Remark \ref{rc2.10}, we deduce that
\begin{align*}
{\rm II}&\leq C\sum_{k=k_0+4}^{\infty}2^{q[k_0(s+1)-k(n+s+1)]}\lf(\frac{|B_k|}{|B_{k_0}|}\r)^{\alpha q/n}\|\chi_{B_{k_0}}\|_{X'}^{q}\|\chi_{B_{k}}\|_{X}^{q}\\
&= C\sum_{k=k_0+4}^{\infty}2^{q[k_0(s+1)-k(n+s+1)]}\lf(\frac{|B_k|}{|B_{k_0}|}\r)^{\alpha q/n}\|\chi_{B_{k_0}}\|_{X'}^{q}\\
&\ \ \ \times\lf(|B_k|\|\chi_{B_{k}}\|_{X'}^{-1}\r)^{q}\lf(\frac{1}{|B_k|}\|\chi_{B_{k}}\|_{X'}\|\chi_{B_{k}}\|_{X}\r)^{q}\\
&\leq C\sum_{k=k_0+4}^{\infty}2^{q(k_0-k)(s+1-\alpha)}\lf(\frac{\|\chi_{B_{k_0}}\|_{X'}}
{\|\chi_{B_{k}}\|_{X'}}\r)^{q}\\
&\leq C\sum_{k=k_0+4}^{\infty}2^{q(k_0-k)(s+1-\alpha+n\delta_{X'})}\leq C.
\end{align*}
Combining the above estimates related to I and II, we obtain the desired estimate for the case $0<q\leq1$.

\textbf{Case 2:} $1<q<\infty$. We assume that $f=\sum_{j\in\zz}\lambda_ja_j \,\,\,\,\mathrm{in}\,\,\,\, \mathcal{S}^{'},$
 where each $a_j$ is a central
$(X,\,\alpha,\,s)$-atom with support contained in $B_j$ and
$$\sum_{j\in\zz}|\lambda_j|^q<\infty.$$
Combining \eqref{e3.2} and \eqref{e2.4xx}, we have
\begin{align*}
\|M_N(f)\|_{\mathcal{H}\dot{\mathcal{K}}_{X}^{\alpha,\,q}(\rn)}^{q}&\leq\sum_{k\in\zz}|B_k|^{\alpha q/n}\lf(\sum_{j\in\zz}|\lambda_j|\lf\|
M_N(a_j)\chi_k\r\|_{X}\r)^q\\
&\leq C \sum_{k\in\zz}|B_k|^{\alpha q/n}\lf(\sum_{j=k-1}^{\infty}|\lambda_j|\lf\|
a_j\r\|_{X}\r)^q\\
&\ \ \ + C\sum_{k\in\zz}|B_k|^{\alpha q/n}\lf(\sum_{j=-\infty}^{k-2}|\lambda_j|\lf\|
M_N(a_j)\chi_k\r\|_{X}\r)^q\\
&=:{\rm \tilde{I}+\tilde{II}}.
\end{align*}
From the size condition of $a_j$ and the H\"{o}lder inequality inequality, it follows that
\begin{align*}
{\rm \tilde{I}}
&\leq C\sum_{k\in\zz}|B_k|^{\alpha q/n}\lf(\sum_{j=k-1}^{\infty}|\lambda_j||B_j|^{-\alpha /n}\r)^q\\
&\leq C\sum_{k\in\zz}|B_k|^{\alpha q/n}
\lf(\sum_{j=k-1}^{\infty}|\lambda_j|^q|B_j|^{-\alpha q/(2n)}\r)\lf(\sum_{j=k-1}^{\infty}|B_j|^{-\alpha q'/(2n)}\r)^{q/q'}\\
&\leq C\sum_{k\in\zz}|B_k|^{\alpha q/(2n)}
\sum_{j=k-1}^{\infty}|\lambda_j|^q|B_j|^{-\alpha q/(2n)}\\
&\leq C
\sum_{j\in\zz}|\lambda_j|^q.
\end{align*}
 Applying a similar estimate of II, the H\"{o}lder inequality and  the fact that $s+1>\alpha-n\delta_{X'}$, we have
\begin{align*}
{\rm \tilde{II}}&\leq C\sum_{k\in\zz}\lf(\sum_{j=-\infty}^{k-2}|\lambda_j|
2^{j(s+1)-k(n+s+1)}\lf(\frac{|B_k|}{|B_{j}|}\r)^{\alpha /n}\|\chi_{B_{j}}\|_{X'}\|\chi_{B_{k}}\|_{X}\r)^{q}\\
&\leq C\sum_{k\in\zz}\lf(\sum_{j=-\infty}^{k-2}|\lambda_j|
2^{(j-k)(s+1-\alpha+n\delta_{X'})}\r)^{q}\\
&\leq C\sum_{k\in\zz}\lf(\sum_{j=-\infty}^{k-2}|\lambda_j|^q
2^{(j-k)(s+1-\alpha+n\delta_{X'})q/2}\r)^{q}\lf(\sum_{j=-\infty}^{k-2}
2^{(j-k)(s+1-\alpha+n\delta_{X'})q'/2}\r)^{q/q'}\\
&\leq C
\sum_{j\in\zz}|\lambda_j|^q.
\end{align*}
This proves the desired estimate for the case $1<q<\infty$. Thus, we obtain that \eqref{e3.3xx} holds true.

\textbf{Step 2.}
 In this step, we prove that
  \begin{equation*}
\mathcal{H}\dot{\mathcal{K}}_{X}^{\alpha,\,q}(\rn)
\subset\mathcal{H}\dot{\mathcal{K}}^{\alpha,\,s,\,q}_{X,\,\rm{atom}}(\rn).
\end{equation*}
  Choosing $\varphi\in\mathcal{C}_0^\infty(\rn)$ such that $\varphi\geq0$, $\int_\rn\varphi(x)\,dx=1$ and supp\,$\varphi\subset\{x:\,|x|\leq1\}$. For any $f\in \mathcal{H}\dot{\mathcal{K}}_{X}^{\alpha,\,q}(\rn)$, set
$f^{(i)}:=f\ast\varphi_{(i)}$, where $ \varphi_{(i)}(\cdot):= 2^{in}\varphi(2^{i}\cdot)$ for any $i\in\zz_
+$. It is  easy to see that $f^{(i)}\in \mathcal{C}^\infty(\rn)$ and $\lim_{i\rightarrow\infty}f^{(i)}\rightarrow f$ in $\mathcal{S}{'}$.
Now we divide \textbf{Step 2} into two substeps.

\textbf{Substep 1.}\,We show that,
for any $x\in\rn$,
$$f^{(i)}(x)=\sum_{j\in\zz}\lambda_{j}a_{j}^{(i)}(x),$$
where $a_{j}^{(i)}$ is a central
$(X,\,\alpha,\,s)$-atom with $\mathrm{supp}\,a_{j}^{(i)}\subset B_{j+2}$,  $\lambda_{j}$ is independent of $i$
and $$\sum_{j\in\zz}|\lambda_j|^q\leq C \|M_Nf\|_ {\dot{K}_{X}^{\alpha,\,q}(\rn)}^{q},$$
with the constant $C > 0$ independent of $f$.

Let $\psi$ be a radial smooth function such that
$0\leq\psi\leq1$, supp\,$\psi \subset \{ x : 1/2- \epsilon \leq |x|\leq 1 +\epsilon \}$ with $0<\epsilon<1/4$ and $\psi(x)=1$ when $ 1/2\leq |x|\leq 1$. Let $\psi_{k}(\cdot):=\psi(2^{-k}\cdot)$ for
$k\in\zz$  and
$$\tilde{D}_{k,\,\epsilon}:=\{x:2^{k-1}-2^{k}\epsilon\leq|x|\leq 2^{k} +2^{k}\epsilon\}.$$
Notice that
{supp}\,$\psi_{k}\subset\tilde{D}_{k,\,\epsilon}$. Let
$$
\Phi_k(x):=\lf\{
\begin{array}{cl}
\frac{\psi_{k}(x)}{\sum_{j\in\zz}\psi_{j}(x)},\,\,\, & \mathrm{if} \ x \neq0,\\
0, \,\,\,& \mathrm{if} \ x =0.
\end{array}\r.
$$
We can obtain that
\begin{align} \label{e3.4xx}
\mathrm{supp}\,\Phi_k\subset\tilde{D}_{k,\,\epsilon}, \,\,0\leq\Phi_k(x)\leq1 \,\,\mathrm{and}\,\, \sum_{k\in\zz}\Phi_k(x)=1 \,\,\mathrm{if}\,\,x\neq0.
\end{align}
For some $s\in\nn$,  denote by $\mathcal{{P}}_s$ the class of all the
real polynomials with the degree less than $s$.
Let $P_k^{(i)}(x):=P_{\tilde{D}_{k,\,\epsilon}}(f^{(i)}\Phi_k)(x)\chi_{\tilde{D}_{k,\,\epsilon}}(x)\in \mathcal{{P}}_s(\rn)$
be the unique polynomial satisfying
$$\int_{\tilde{D}_{k,\,\epsilon}}\lf(f^{(i)}\Phi_k)(x)-P_k^{(i)}(x)\r)x^\beta dx=0,\ \ \  |\beta|\leq s:=\lfloor\alpha-n\delta_{X'}\rfloor.$$
Then we have
\begin{align*}
f^{(i)}(x)&=f^{(i)}(x)\sum_{k\in\zz}\Phi_k(x)\\
&=\sum_{k\in\zz}\lf(f^{(i)}(x)\Phi_k(x)-P_k^{(i)}(x)\r)+\sum_{k\in\zz}P_k^{(i)}(x)\\
&=:{\rm I}_1^{(i)}+{\rm I}_2^{(i)}.
\end{align*}
To deal with ${\rm I}_1^{(i)}$, let
$$g_k^{(i)}(x):=f^{(i)}(x)\Phi_k(x)-P_k^{(i)}(x)$$
and
$$a_{k}^{(i)}(x):=\frac{g_k^{(i)}(x)}{\lambda_{k}},\,\,\,
\lambda_{k}:=C_1|B_{k+1}|^{\alpha/n}\sum_{l=k-1}^{k+1}\|M_Nf\chi_l\|_{X},$$
where $C_1$ is a constant which will be chosen later. Then we know that
$$\mathrm{supp}\,a_{k}^{(i)}\subset B_{k+1}\,\,{\rm and}\,\,\int_{\rn}a_{k}^{(i)}(x)x^\beta\,dx=0
,\,\, |\beta|\leq s:=\lfloor\alpha-n\delta_{X'}\rfloor. $$
Moreover,
$${\rm I}_1^{(i)}=\sum_{k\in\zz}\lambda_{k}a_{k}^{(i)}(x).$$

Now let's  estimate $\|g_k^{(i)}\|_{X}$. Let $\{\phi_\gamma^{k}:|\gamma|\leq s\}$ be the orthogonal
polynomials restricted to $\tilde{D}_{k,\,\epsilon}$ with respect to the weight $1/|\tilde{D}_{k,\,\epsilon}|$, which are obtained
from $\{x^\beta:|\beta|\leq s\}$ by Gram-Schmidt's method, which means
\begin{align}\label{e3.4x'}
\lf\langle\phi_\nu^{k},\,\phi_\mu^{k}\r\rangle=\frac{1}{|\tilde{D}_{k,\,\epsilon}|}
\int_{\tilde{D}_{k,\,\epsilon}}\phi_\nu^{k}(x)\,\phi_\mu^{k}(x)dx=\delta_{\nu\,\mu},
\end{align}
where $\delta_{\nu\,\mu}=1$ for $\nu=\mu$, otherwise $0$.

It is easy to see that
$$P_k^{(i)}(x)=\sum_{|\gamma|\leq s}\lf\langle f^{(i)}\Phi_k,\,\phi_\gamma^{k}\r\rangle\phi_\gamma^{k}(x),\ \ x\in\tilde{D}_{k,\,\epsilon}.$$
From \eqref{e3.4x'}, we deduce that
\begin{align*}
\frac{1}{|\tilde{D}_{1,\,\epsilon}|}
\int_{\tilde{D}_{1,\,\epsilon}}\phi_\nu^{k}(2^{k-1}y),\,\phi_\mu^{k}((2^{k-1}y))dy=\delta_{\nu\,\mu}.
\end{align*}
Thus we have $\phi_\nu^{k}(2^{k-1}y)=\phi_\nu^{1}(y)$ $a.e.$.  That is $\phi_\nu^{k}(x)=\phi_\nu^{1}({2^{1-k}}x)$ $a.e.$ for any $x\in\tilde{D}_{k,\,\epsilon}$. Therefore, we have $|\phi_\nu^{k}(x)|\leq C$. For any $x\in\tilde{D}_{k,\,\epsilon}$, by Lemma \ref{l3.4x}, we obtain that
\begin{align*}
|P_k^{(i)}(x)|&\leq\frac{C}{|\tilde{D}_{k,\,\epsilon}|}\int_{\tilde{D}_{k,\,\epsilon}}f^{(i)}(x)\Phi_k(x)dx\\
&\leq\frac{C}{|\tilde{D}_{k,\,\epsilon}|}\|f^{(i)}\Phi_k\|_{X}\|\chi_{\tilde{D}_{k,\,\epsilon}}\|_{X'}.
\end{align*}
From this, \eqref{e3.2}, Lemma \ref{l3.5x} and \eqref{e3.4xx}, it follows that
\begin{align*}
\|g_k^{(i)}\|_{X}&\leq\|f^{(i)}\Phi_k\|_{X}+\|P_k^{(i)}\|_{X}\\
&\leq \|f^{(i)}\Phi_k\|_{X}+
\frac{C}{|\tilde{D}_{k,\,\epsilon}|}\|f^{(i)}\Phi_k\|_{X}\|\chi_{\tilde{D}_{k,\,\epsilon}}\|_{X'}
\|\chi_{\tilde{D}_{k,\,\epsilon}}\|_{X}\\
&\leq \|f^{(i)}\Phi_k\|_{X}+C\|f^{(i)}\Phi_k\|_{X}\\
&\leq C'\sum_{l=k-1}^{k+1}\|M_Nf\chi_l\|_{X}.
\end{align*}
Take $C_1=C^{'}$, then
$\|a_{k}^{(i)}\|_{X}\leq|B_{k+1}|^{-\alpha/n}$, and hence $a_{k}^{(i)}$ is a central
$(X,\,\alpha,\,s)$-atom  with $\mathrm{supp}\,a_{k}^{(i)}\subset B_{k+1}$. Moreover,
\begin{align*}
\sum_{k\in\zz}|\lambda_{k}|^q\leq C\sum_{k\in\zz}|B_{k+1}|^{\alpha q/n}\lf(\sum_{l=k-1}^{k+1}
\|M_Nf\chi_l\|_{X}\r)^q
\leq C\|M_Nf\|_ {\dot{K}_{X}^{\alpha,\,q}(\rn)}<\infty,
\end{align*}
where $C$ is independent of $i$ and $f$.

Next we deal with ${
\rm I}_2^{(i)}$. Let $\{\psi_\gamma^k : |\gamma|\leq s\}$ be the dual basis of $\{x^\beta:|\beta|\leq s\}$ with respect to the weight  $1/|\tilde{D}_{k,\,\epsilon}|$ on $\tilde{D}_{k,\,\epsilon}$, that is
\begin{align*}
\lf\langle \psi_\gamma^k,\,x^\beta\r\rangle=\frac{1}{|\tilde{D}_{k,\,\epsilon}|}
\int_{\tilde{D}_{k,\,\epsilon}}\psi_\gamma^k(x)\,x^\beta dx=\delta_{\beta\,\gamma}.
\end{align*}
By an argument similar to that used in the proof of \cite{ly95}, let
 \begin{align*}
h_{k,\,\gamma}^{(i)}(x):=\sum_{l=-\infty}^{k}\lf( \frac{\psi_\gamma^k(x)\,\chi_{\tilde{D}_{k,\,\epsilon}}(x)}{|\tilde{D}_{k,\,\epsilon}|}
-\frac{\psi_\gamma^{k+1}(x)\,\chi_{\tilde{D}_{k+1,\,\epsilon}}(x)}{|\tilde{D}_{k+1,\,\epsilon}|}\r)
\int_{\rn}f^{(i)}(x)\Phi_l(x)dx,
\end{align*}
and we have
$$\int_{\rn}h_{k,\,\gamma}^{(i)}(x)\,x^\beta dx=0,\ \ \  |\beta|\leq s=\lfloor\alpha-n\delta_{X'}\rfloor.$$
Then we can write
\begin{align*}
{\rm I}_2^{(i)}&=\sum_{k\in\zz}\sum_{|\gamma|\leq m}\lf\langle f^{(i)}\Phi_k,\,x^\gamma\r\rangle
\frac{\psi_\gamma^k(x)\,\chi_{\tilde{D}_{k,\,\epsilon}}(x)}{|\tilde{D}_{k,\,\epsilon}|}\\
&=\sum_{|\gamma|\leq m}\sum_{k\in\zz}\lf(\int_{\rn}f^{(i)}(x)\Phi_k(x)\,x^\gamma dx\r)
\frac{\psi_\gamma^k(x)\,\chi_{\tilde{D}_{k,\,\epsilon}}(x)}{|\tilde{D}_{k,\,\epsilon}|}\\
&=\sum_{|\gamma|\leq m}\sum_{k\in\zz}\lf(\sum_{l=-\infty}^{k}\int_{\rn}f^{(i)}(x)\Phi_l(x)\,x^\gamma dx\r)\\
&\ \ \ \times\lf( \frac{\psi_\gamma^k(x)\,\chi_{\tilde{D}_{k,\,\epsilon}}(x)}{|\tilde{D}_{k,\,\epsilon}|}
-\frac{\psi_\gamma^{k+1}(x)\,\chi_{\tilde{D}_{k+1,\,\epsilon}}(x)}{|\tilde{D}_{k+1,\,\epsilon}|}\r)\\
&=\sum_{|\gamma|\leq m}\sum_{k\in\zz}h_{k,\,\gamma}^{(i)}(x).
\end{align*}
Let $a_{k,\,\gamma}^{(i)}:=h_{k,\,\gamma}^{(i)}(x)/\lambda_{k,\,\gamma}$ and $$\lambda_{k,\,\gamma}:=C_2b^{(k+2)\alpha_{k+2}}\sum_{l=k-1}^{k+2}\|M_Nf\chi_l\|_{X},$$
where $C_2$ is a constant to be determined later. Notice that
$$\mathrm{supp}\,a_{k,\,d}^{(i)}\subset\tilde{D}_{k,\,\epsilon}\cup \tilde{D}_{k+1,\,\epsilon} \subset B_{k+2}\,\,{\rm and}\,\,\int_{\rn}a_{k,\,\gamma}^{(i)}(x)x^\beta\,dx=0
,\,\, |\beta|\leq s:=\lfloor\alpha-n\delta_{X'}\rfloor.$$
Moreover,
$${\rm I}_2^{(i)}=\sum_{k\in\zz}\lambda_{k,\,d}\,a_{k,\,\gamma}^{(i)}(x).$$

Now we estimate  $\|h_{k,\,\gamma}^{(i)}\|_{X}$. From \eqref{e3.4xx}, we conclude that
$$\int_{\rn}\sum_{l=-\infty}^{k}|\Phi_l(x)\,x^\gamma| dx\leq
\sum_{l=-\infty}^{k}\int_{\tilde{D}_{k,\,\epsilon}}|\Phi_l(x)\,x^\gamma| dx \leq C 2^{k(n+|d|)}.$$
Thus, by a computation, we obtain that
$$\lf|\int_{\rn}f^{(i)}(y)\sum_{l=-\infty}^{k}\Phi_l(y)\,y^\gamma dx\r|\leq C 2^{k(n+|d|)}M_Nf(x),\,\,x\in B_{k+2},$$
which, together with the fact that
 $$\lf( \frac{\psi_\gamma^k(x)\,\chi_{\tilde{D}_{k,\,\epsilon}}(x)}{|\tilde{D}_{k,\,\epsilon}|}
-\frac{\psi_\gamma^{k+1}(x)\,\chi_{\tilde{D}_{k+1,\,\epsilon}}(x)}{|\tilde{D}_{k+1,\,\epsilon}|}\r)\leq C 2^{k(n+|d|)}\sum_{l=k-1}^{k+1}\chi_l(x),$$
further implies that
$$\lf\|h_{k,\,\gamma}^{(i)}\r\|_{X}\leq C''\,2^{k(n+|d|)}\sum_{l=k-1}^{k+1}\|M_Nf\chi_l\|_{X},$$
where $ C''$ is a constant independent of $i,\, f, k$ and $\gamma$.
Choose $C_2=C''$. It is easy to see that $a_{k,\,\gamma}^{(i)}$ is a central
$(X,\,\alpha,\,s)$-atom with $\mathrm{supp}\,a_{k,\,\gamma}^{(i)}\subset B_{k+2}$.
Furthermore,
\begin{align*}
\sum_{k,\,\gamma}|\lambda_{k,\,\gamma}|^q\leq C\sum_{k\in\zz}|B_{k+2}|^{\alpha q/n}\lf(\sum_{l=k-1}^{k+1}
\|M_Nf\chi_l\|_{X}\r)^q
\leq C\|M_Nf\|_ {\dot{\mathcal{K}}_{X}^{\alpha,\,q}(\rn)}<\infty.
\end{align*}

Therefore, we  conclude that for any $x\in\rn$,
$$f^{(i)}(x)=\sum_{j\in\zz}\lambda_{j}a_{j}^{(i)}(x),$$
where $a_{j}^{(i)}$ is a central
$(X,\,\alpha,\,s)$-atom with $\mathrm{supp}\,a_{j}^{(i)}\subset B_{k+2}$,  $\lambda_{j}$ is independent of $i$
and$$\sum_{j\in \zz}|\lambda_j|^q\lesssim\|M_Nf\|_ {\dot{\mathcal{K}}_{X}^{\alpha,\,q}(\rn)}^{q}<\infty.$$

Notice that
$$\sup_{i\in\zz_+}\|a_{0}^{(i)}\|_{X}\leq|B_2|^{-\alpha/n}.$$
Combining the Banach-Alaoglu theorem, we obtain a subsequence $\{a_{0}^{(i_{n_0})}\}$ of $\{a_{0}^{(i)}\}$ converging in the w$^\ast$ topology of $X$ to some $a_0\in X$. It is obvious that $a_0$ is a a central
$(X,\,\alpha,\,s)$-atom with $\mathrm{supp}\,a_0\subset B_{2}$. Next, since
$$\sup_{i_{n_0}\in\zz_+}\|a_{0}^{(i_{n_0})}\|_{X}\leq|B_3|^{-\alpha/n},$$
applying Banach-Alaoglu theorem, we obtain that there exists a subsequent $\{a_{1}^{(i_{n_1})}\}$ of $\{a_{1}^{(i_{n_0})}\}$ converging in the w$^\ast$ topology of $X$ to a central
$(X,\,\alpha,\,s)$-atom $a_1$ with $\mathrm{supp}\,a_1\subset B_{3}$. Repeating the above procedure for any $j\in\zz$, we can find a subsequence $\{a_{j}^{(i_{n_j})}\}$ of $\{a_{j}^{(i)}\}$ converging in the w$^\ast$ topology of $X$ to a central
$(X,\,\alpha,\,s)$-atom $a_j$ with $\mathrm{supp}\,a_j\subset B_{j+2}$. By usual diagonal method we get a subsequence $\{i_\nu\}$ of $\zz_+$
such that for any $j\in\nn$, $\lim_{\nu\rightarrow\infty}a_{j}^{(i_{\nu})}=a_j$ in the w$^\ast$ topology of $X$ and therefore in $\mathcal{S}^{'}$.

\textbf{Substep 2.}
In this substep, we prove
\begin{align}\label{e2.7}
f=\sum_{j\in\zz}\lambda_ja_j\,\,\mathrm{in}\,\,\mathcal{S}^{'}.
\end{align}
 Observe that
\begin{align}\label{e3.7xx}
\mathrm{supp}\,a_{j}^{(i_{\nu})}\subset(\tilde{D}_{j,\,\epsilon}\cup \tilde{D}_{j+1,\,\epsilon})
\subset (D_{j-1}\cup D_{j}\cup D_{j+1}\cup D_{j+2})\subset B_{j+2}.
\end{align}
Then for any $\phi\in\mathcal{S}$, we have
\begin{align}\label{e3.8xx}
\langle f,\,\phi\rangle=\lim_{\nu\rightarrow\infty}\sum_{j\in\zz}\lambda_j\int_{\rn}a_{j}^{(i_{\nu})}(x)\phi(x)\,dx
\ \ \ (\text{see \cite{ly95} for more details}).\end{align}

 Notice that $s=\lfloor\alpha-n\delta_{X'}\rfloor$. If $j\leq0$, then, by the atom conditions of $a_{j}^{(i_{\nu})}$, Taylor Remainder Theorem, Lemma \ref{l3.4x} and Remark \ref{rc2.10}, we obtain that
\begin{align*}
\lf|\int_{\rn}a_{j}^{(i_{\nu})}(x)\phi(x)\,dx\r|&=\lf|\int_{\rn}a_{j}^{(i_{\nu})}(x)\lf(\phi(x)-\sum_{|\beta|\leq s}\frac{\partial^{\beta}\phi(0)}{\beta!}x^{\beta}\r)\,dx\r|\\
&\leq C \int_{B_{j+2}}\lf|a_{j}^{(i_{\nu})}(x)\r||x|^{s+1}\,dx\\
&\leq C\, 2^{j(s+1)} \int_{\rn}\lf|a_{j}^{(i_{\nu})}(x)\r|\chi_{B_{j+2}}(x)\,dx\\
&\leq C\, 2^{j(s+1)}\lf\|a_{j}^{(i_{\nu})}\r\|_{X}\|\chi_{B_{j+2}}\|_{X'}\\
&\leq C\, 2^{j(s+1-\alpha)}\lf(\frac{|B_{j+2}|}{|B_{2}|}\r)^{\delta_{X'}}\|\chi_{B_{2}}\|_{X'}\\
&\leq C\, 2^{j(s+1-\alpha+n\delta_{X'})}\frac{|B_{2}|}{|B_{0}|}\|\chi_{B_{0}}\|_{X'}\\
&\leq C\, 2^{j(s+1-\alpha+n\delta_{X'})}.
\end{align*}
If $j>0$, choose $k_0\in\zz_+$ such that $k_0+\alpha-n>0$, then by \eqref{e3.7xx}, Lemma \ref{l3.4x}, the size condition of $a_{j}^{(i_{\nu})}$ and Remark \ref{rc2.10}, we get
\begin{align*}
\lf|\int_{\rn}a_{j}^{(i_{\nu})}(x)\phi(x)\,dx\r|&\leq C\,\int_{\rn}\lf|a_{j}^{(i_{\nu})}(x)\r||x|^{-k_0}\,dx\\
&\leq C\, 2^{-jk_0}\lf\|a_{j}^{(i_{\nu})}\r\|_{X}\|\chi_{B_{j+2}}\|_{X'}\\
&\leq C\, 2^{-j(k_0+\alpha)}\frac{|B_{j+2}|}{|B_{0}|}\|\chi_{B_{0}}\|_{X'}\\
&\leq C\, 2^{-j(k_0+\alpha-n)},
\end{align*}
where $C$ is independent of $j$.

Let
$$
\mu_j:=\lf\{
\begin{array}{cl}
\lf|\lambda_j\r|2^{j(s+1-\alpha+n\delta_{X'})},\,\,\, &  \ j \leq0,\\
|\lambda_j|2^{-j(k_0+\alpha-n)}, &\ j >0.
\end{array}\r.
$$
Then, we have
$$\sum_{j=-\infty}^{\infty}|\mu_j|\leq C\lf(\sum_{j=-\infty}^{\infty}|\lambda_j|^q\r)^{1/q}\leq C\|M_Nf\|_ {\dot{K}_{X}^{\alpha,\,q}(\rn)}$$
and
$$|\lambda_j|\lf|\int_{\rn}a_{j}^{(i_{\nu})}(x)\phi(x)\,dx\r|\leq C|\mu_j|,
$$
which, together with \eqref{e3.8xx},  further implies that
$$\langle f,\,\phi\rangle=\sum_{j\in\zz}\lim_{\nu\rightarrow\infty}\lambda_j\int_{\rn}a_{j}^{(i_{\nu})}(x)\phi(x)\,dx
=\sum_{j\in\zz}\lambda_j\int_{\rn}a_{j}(x)\phi(x)\,dx,$$
and hence \eqref{e2.7} holds true. Thus, we have
 \begin{equation*}
\mathcal{H}\dot{\mathcal{K}}_{X}^{\alpha,\,q}(\rn)
\subset\mathcal{H}\dot{\mathcal{K}}^{\alpha,\,s,\,q}_{X,\,\rm{atom}}(\rn).
\end{equation*}This finishes the proof of Theorem \ref{t3.1}.
\end{proof}
\begin{remark}\label{r4.4}
If $f\in\mathcal{H}\dot{\mathcal{K}}_{X}^{\alpha,\,q}(\rn)\cap \mathcal{S}$, we can replace $f^{(i)}$ by $f$ in the proof
of Step 2 , and we further obtain
\begin{align*}
f(x)=\sum_{j\in\zz}\lambda_ja_j(x)+\sum_{j\in\zz}\mu_jb_j(x),
\end{align*}
where
\begin{align*}
\|a_j\|_{X}\leq \lf|B(0,\,2^{j+1})\r|^{-\alpha/n},\,\, \|b_j\|_{X}\leq \lf|B(0,\,2^{j+2})\r|^{-\alpha/n},\\
\supp a_j \subset \tilde{D}_{j,\,\epsilon},\,\, \supp b_j \subset \tilde{D}_{j,\,\epsilon}\cup\tilde{D}_{j+1,\,\epsilon}
\end{align*}
and
\begin{align*}
0\leq\lambda_j,\,\mu_j\leq C 2^{\alpha j}\sum_{k=j-1}^{j+1}\|M_Nf\chi_k\|_{X}
\end{align*}
with $N>n+\alpha+1$.
\end{remark}

As an application of Theorem \ref{t3.1}, we obtain the boundedness of certain sublinear operators from $\mathcal{H}\dot{\mathcal{K}}_{X}^{\alpha,\,q}(\rn)$ to $\dot{\mathcal{K}}_{X}^{\alpha,\,q}(\rn)$  as follows.

\begin{theorem}\label{t5.1}
Let $X$ be a ball Banach function space and satisfy  Assumption\,1, and let $q\in(0,\,\infty)$, $\alpha \in [n\delta_{X'},\,\infty)$ and $s\in[(\alpha-n\delta_{X'}),\,\infty)\cap\zz_+$. If a sublinear operator $T$ satisfies that
\begin{enumerate}
\item[\rm{(i)}]
$T$ is bounded on $X$;
\item[\rm{(ii)}]
there exists a constant $\delta>0$ such that $s+\delta>\alpha-n\delta_{X'}$, and for any compact
support function $f$ with
$$\int_{\rn}f(x)\,x^\beta dx=0,\,\,\,|\beta|\leq s,$$
$T(f)$ satisfies the size condition
\begin{align}\label{e}
|T(f)(x)|&\leq C ({\rm diam}\,({\rm supp} \,f))^{s+\delta} |x|^{-(n+s+\delta)} \|f\|_{L^1{(\rn)}},\\ &\,\,\,\mathrm{if}\,\, {\rm dist}\,(x,\,{\rm supp}\,f)\geq\frac{|x|}{2}\nonumber.
\end{align}
 \end{enumerate}
 Then there exists a positive constants $C_1$  independent of $f$ such that,  for any
 $f\in \mathcal{H}\dot{\mathcal{K}}_{X}^{\alpha,\,q}(\rn)$,
 $$\|T(f)\|_{\dot{\mathcal{K}}_{X}^{\alpha,\,q}(\rn)}\leq C_1\|f\|_{\mathcal{H}\dot{\mathcal{K}}_{X}^{\alpha,\,q}(\rn)}.$$
\end{theorem}
\begin{proof}
 Let $f\in \mathcal{H}\dot{\mathcal{K}}_{X}^{\alpha,\,q}(\rn)$. From Theorem \ref{t3.1},
we know that there exist $\{\lambda_j\}_{j\in\nn}\subset\ccc$ and a sequence of central $(X,\,\alpha,\,s)$-atoms, $\{a_j\}_{j\in\zz}$, supported, respectively,
on $\{{B_j}\}_{j\in\zz}\subset\mathfrak{B}$ such that
\begin{align*}
f=\sum_{j\in\zz} \lz_{j}a_j \ \ \mathrm{in\ } \ \cs'
\end{align*}
and
\begin{eqnarray*}
\|f\|_{\mathcal{H}\dot{\mathcal{K}}_{X}^{\alpha,\,q}(\rn)}\sim \lf(\sum_{j\in\zz_+}|\lambda_j|^q\r)^{1/q}<\infty.
\end{eqnarray*}
Thus, we have
\begin{align*}
\|T(f)\|_{\dot{\mathcal{K}}_{X}^{\alpha,\,q}(\rn)}^q&=\sum_{k\in\zz}2^{k\alpha q}\lf\|T(f)\chi_k\r\|
_{X}^q\\
&\leq C \sum_{k\in\zz}2^{k\alpha q}\lf(\sum_{j=-\infty}^{k-2}|\lambda_j|\lf\|
T(a_j)\chi_k\r\|_{X}\r)^q\\
&\ \ \ + C\sum_{k\in\zz}2^{k\alpha q}\lf(\sum_{j=k-1}^{\infty}|\lambda_j|\lf\|
T(a_j)\chi_k\r\|_{X}\r)^q\\
&=:{\rm S_1+S_2}.
\end{align*}
To deal with ${\rm S_1}$, we first estimate $\|
T(a_j)\|_{X}$.
For any $x\in D_k$ and $j\leq k-2$, by \eqref{e}, Lemma \ref{l3.4x} and the size condition of $a_j$, we obtain that
\begin{align*}
|T(a_j)(x)|&\leq C 2^{j(s+\delta)}|x|^{-(n+s+\delta)}\int_{B_{j}}\lf|a_{j}(y)\r|\,dy\\
&\leq C\, 2^{j(s+\delta)}2^{-k(n+s+\delta)} \int_{\rn}\lf|a_{j}(x)\r|\chi_{B_{j}}(x)\,dx\\
&\leq C\,2^{j(s+\delta)}2^{-k(n+s+\delta)}\lf\|a_{j}\r\|_{X}\|\chi_{B_{j}}\|_{X'}\\
&\leq C\,2^{j(s+\delta-\alpha)-k(n+s+\delta)}\|\chi_{B_{j}}\|_{X'}.
\end{align*}
From this, Lemma \ref{l3.4x}, \eqref{e3.2}, Lemma \ref{l3.5x} and Remark \ref{rc2.10}, it follows that
\begin{align}\label{e4.3xx}
\lf\|
T(a_j)\chi_k\r\|_{X}&\leq C\,2^{j(s+\delta-\alpha)-k(n+s+\delta)}\|\chi_{B_{j}}\|_{X'}\|\chi_{B_{k}}\|_{X}\\
&= C\,2^{j(s+\delta-\alpha)-k(n+s+\delta)}\|\chi_{B_{j}}\|_{X'}\nonumber\\
&\ \ \ \times\lf(|B_k|\|\chi_{B_{k}}\|_{X'}^{-1}\r)\lf(\frac{1}{|B_k|}\|\chi_{B_{k}}\|_{X'}\|\chi_{B_{k}}\|_{X}\r)\nonumber\\
&\leq C\,2^{j(s+\delta-\alpha)-k(s+\delta)}\frac{\|\chi_{B_{j}}\|_{X'}}
{\|\chi_{B_{k}}\|_{X'}}\nonumber\\
&\leq C\,2^{(s+\delta+n\delta_{X'})(j-k)-j\alpha}\nonumber.
\end{align}
Now, let's estimate ${\rm S_1}$ in the following two cases.

\textbf{Case 1:} $0<q\leq1$. From \eqref{e4.3xx} and  $n\delta_{X'}\leq\alpha<s+\delta+n\delta_{X'}$, we deduce that
\begin{align*}
{\rm S_1}&=C \sum_{k\in\zz}2^{k\alpha q}\lf(\sum_{j=-\infty}^{k-2}|\lambda_j|\lf\|
T(a_j)\chi_k\r\|_{X}\r)^q\\
&\leq C \sum_{k\in\zz}2^{k\alpha q}\lf(\sum_{j=-\infty}^{k-2}|\lambda_j|^q 2^{[(s+\delta+n\delta_{X'})(j-k)-j\alpha]q}\r)\\
&\leq C \sum_{j\in\zz}|\lambda_j|^q \sum_{k=j+2}^{\infty} 2^{(s+\delta+n\delta_{X'}-\alpha)(j-k)q}\\
&\leq C \sum_{j\in\zz}|\lambda_j|^q.
\end{align*}

\textbf{Case 2:} $1<q\leq\infty$. Since  $n\delta_{X'}\leq\alpha<s+\delta+n\delta_{X'}$, by \eqref{e4.3xx} and  the H\"{o}lder inequality, we obtain that
\begin{align*}
{\rm S_1}
&\leq C \sum_{k\in\zz}2^{k\alpha q}\lf(\sum_{j=-\infty}^{k-2}|\lambda_j| 2^{(s+\delta+n\delta_{X'})(j-k)-j\alpha}\r)^q\\
&\leq C \sum_{k\in\zz}\lf(\sum_{j=-\infty}^{k-2}|\lambda_j|^q 2^{(s+\delta+n\delta_{X'}-\alpha)(j-k)q/2}\r)\\
&\ \ \ \times\lf(\sum_{j=-\infty}^{k-2}2^{(s+\delta+n\delta_{X'}-\alpha)(j-k)q'/2}\r)^{q/q'}\\
&\leq C \sum_{j\in\zz}|\lambda_j|^q \sum_{k=j+2}^{\infty} 2^{(s+\delta+n\delta_{X'}-\alpha)(j-k)q/2}\\
&\leq C \sum_{j\in\zz}|\lambda_j|^q.
\end{align*}

Next, let's estimate ${\rm S_2}$. Similarly to the estimate of ${\rm S_1}$, we consider two cases.

\textbf{Case A:} $0<q\leq1$. By  the boundedness of $T$ on $X$,  the size condition of $a_j$, we conclude that
\begin{align*}
{\rm S_2}&=C\sum_{k\in\zz}2^{k\alpha q}\lf(\sum_{j=k-1}^{\infty}|\lambda_j|\lf\|
T(a_j)\chi_k\r\|_{X}\r)^q\\
&\leq C\sum_{k\in\zz}2^{k\alpha q}\lf(\sum_{j=k-1}^{\infty}|\lambda_j|^q\lf\|
a_j\r\|_{X}^q\r)\\
&\leq C\sum_{k\in\zz}2^{k\alpha q}\lf(\sum_{j=k-1}^{\infty}|\lambda_j|^q\lf|B_{j}\r|^{-\alpha q/n}\r)\\
&\leq C \sum_{j\in\zz}|\lambda_j|^q \sum_{k=-\infty}^{j+1} 2^{(k-j)\alpha q}\\
&\leq C \sum_{j\in\zz}|\lambda_j|^q.
\end{align*}

\textbf{Case B:} $1<q\leq\infty$.  By  the H\"{o}lder inequality,  boundedness of $T$ on $X$,  the size condition of $a_j$, we have
\begin{align*}
{\rm S_2}&\leq C\sum_{k\in\zz}2^{k\alpha q}\lf(\sum_{j=k-1}^{\infty}|\lambda_j|^q\lf\|
T(a_j)\chi_k\r\|_{X}^{q/2}\r)\lf(\sum_{j=k-1}^{\infty}\lf\|
T(a_j)\chi_k\r\|_{X}^{q'/2}\r)^{q/q'}\\
&\leq C\sum_{k\in\zz}2^{k\alpha q}\lf(\sum_{j=k-1}^{\infty}|\lambda_j|^q\lf\|
a_j\r\|_{X}^q\r)\\
&\leq C\sum_{k\in\zz}2^{k\alpha q}\lf(\sum_{j=k-1}^{\infty}|\lambda_j|^q\lf\|
a_j\r\|_{X}^{q/2}\r)\lf(\sum_{j=k-1}^{\infty}\lf\|
a_j\r\|_{X}^{q'/2}\r)^{q/q'}\\
&\leq C\sum_{k\in\zz}2^{k\alpha q}\lf(\sum_{j=k-1}^{\infty}|\lambda_j|^q\lf|B_j\r|^{-\alpha q/(2n)}\r)\lf(\sum_{j=k-1}^{\infty}\lf|B_j\r|^{-\alpha q'/(2n)}\r)^{q/q'}\\
&\leq C\sum_{k\in\zz}2^{k\alpha q/2}\lf(\sum_{j=k-1}^{\infty}|\lambda_j|^q\lf|B_j\r|^{-\alpha q/(2n)}\r)\\
&\leq C \sum_{j\in\zz}|\lambda_j|^q \sum_{k=-\infty}^{j+1} 2^{(k-j)\alpha q/2}\\
&\leq C \sum_{j\in\zz}|\lambda_j|^q.
\end{align*}
Combining the above estimates related to ${\rm S_1}$ and ${\rm S_2}$,  we conclude  that
$$\|T(f)\|_{\dot{\mathcal{K}}_{X}^{\alpha,\,q}(\rn)}\leq C\|f\|_{\mathcal{H}\dot{\mathcal{K}}_{X}^{\alpha,\,q}(\rn)},$$
where $C>0$ is independent of $f$. This finishes the proof Theorem \ref{t5.1}.
\end{proof}

\section{Three examples of Herz-type Hardy spaces associated with ball quasi-Banach function spaces}\label{s7}
\hskip\parindent
In this section, we apply Theorems \ref{t3.1} and \ref{t5.1} to three specific examples: Herz-type Hardy spaces with variable exponents (detailed in Subsection \ref{s4.1}), mixed Herz-type Hardy spaces (detailed in Subsection \ref{s4.2}) and  Herz-Orlicz Hardy spaces (detailed in Subsection \ref{s4.3}). These examples are all components of the family of Herz-type Hardy spaces associated with ball quasi-Banach function spaces. It is pointing out that, even in these cases, the results herein  include the corresponding results for the mixed Herz Hardy space and the Herz-type Hardy spaces with variable exponents as special cases. Moreover, Herz-Orlicz Hardy spaces have not been previously explored. Thus, the results herein  apply to Herz-Orlicz Hardy spaces for the first time.
\subsection{Herz-type Hardy spaces with variable
exponent}\label{s4.1}

In this subsection, we show that the variable Lebesgue space is a ball Banach function space. Therefore, Herz-type Hardy spaces with variable
exponent (see \cite{wl12}) belong to the family of Herz-type Hardy spaces associated with ball quasi-Banach function spaces. Moreover, the atomic decompositions of $\mathcal{H}\dot{\mathcal{K}}_{X}^{\alpha,\,q}(\mathbb{R}^n)$ (see Theorem \ref{t3.1}) and the boundedness of certain sublinear operators from $\mathcal{H}\dot{\mathcal{K}}_{X}^{\alpha,\,q}(\mathbb{R}^n)$ to $\dot{\mathcal{K}}_{X}^{\alpha,\,q}(\mathbb{R}^n)$ (see Theorem \ref{t5.1}) can be applied to the aforementioned Herz-type Hardy spaces with variable
exponent, i.e., the following Theorem \ref{ta4.1x}. To prove this, we need to recall several definitions.

For any measurable function  $p(\cdot):\mathbb{R}^n\rightarrow(0,\,\infty)$,
the {\it variable Lebesgue spaces} $L^{p(\cdot)}(\mathbb{R}^{n})$ denotes the set of
measurable functions $f$ on $\mathbb{R}^n$ such that
\begin{equation*}
\|f\|_{L^{p(\cdot)}(\rn)}:=\inf\lf\{ \lz\in(0,\,\fz):\ \int_{\mathbb{R}^n}\lf(\frac{|f(x)|}
{\lambda}\r)^{p(x)}\,dx\le 1\r\}<\infty.
\end{equation*}
More  properties about $L^{p(\cdot)}$ are referred to \cite{cf13,dhh11}.

Define \(\mathscr{P}(\rn)\) to be the set of measurable functions $p(\cdot):\mathbb{R}^n\rightarrow(0,\,\infty)$ such that
\begin{equation*}
p_{-}:=\mathop\mathrm{\,ess\,inf\,}_{x\in \mathbb{R}^n}p(x)>1 \ \ \ {\rm and} \qquad  p_{+}:=\mathop\mathrm{\,ess\,sup\,}_{x\in \mathbb{R}^n}p(x)<\infty
.
\end{equation*}
Let \(\mathscr{B}(\rn)\) be the set of \(p(\cdot) \in \mathscr{P}(\rn)\) such that the Hardy-Littlewood maximal operator $M_{\mathrm{HL}}$ is bounded on \(L^{p(\cdot)}(\rn)\).

By  \cite[Section 5.3]{wy19}, we know that $L^{p(\cdot)}(\rn)$ is a ball
quasi-Banach function space if $0<p_{-}\leq p_{+}<\infty$, and $L^{p(\cdot)}(\rn)$ is a ball
Banach function space if $1<p_{-}\leq p_{+}<\infty$. Let $p(\cdot)\in \mathscr{P}(\rn)$ and let $X=L^{p(\cdot)}(\rn)$. For any given $q\in(0,\,\infty)$ and  $\alpha\in \mathbb{R}$, $\mathcal{H}\dot{\mathcal{K}}_{X}^{\alpha,\,q}(\rn)$ is the Herz-type Hardy space with variable
exponent $\mathcal{H}\dot{\mathcal{K}}_{L^{p(\cdot)}}^{\alpha,\,q}(\rn)$ (see \cite[Definition\,2.1]{wl12}). By applying Theorems \ref{t3.1} and  \ref{t5.1} to $X=L^{p(\cdot)}(\rn)$, we obtain the following theorem.

\begin{theorem}\label{ta4.1x}
 For any \(p(\cdot) \in \mathscr{B}(\rn)\), Theorems \ref{t3.1} and \ref{t5.1} hold true with $X$ replaced by $L^{p(\cdot)}(\rn)$.
\end{theorem}

\begin{remark}
   Let \(p(\cdot) \in \mathscr{B}(\rn)\). When $X=L^{p(\cdot)}(\rn)$,  for any  $q\in(0,\,\infty)$ and  $\alpha\in \mathbb{R}$, we have  $\mathcal{H}\dot{\mathcal{K}}_{X}^{\alpha,\,q}(\rn)=\mathcal{H}\dot{\mathcal{K}}_{L^{p(\cdot)}}^{\alpha,\,q}(\rn)$. In this case,  Theorems \ref{t3.1} and \ref{t5.1} coincide with  \cite[Theorems 2.1 and 2.2]{wl12}.
\end{remark}

\begin{proof}[Proof of Theorem \ref{ta4.1x}]
Let \(p(\cdot) \in \mathscr{B}(\rn)\). Since the space $L^{p(\cdot)}(\rn)$ is a ball Banach function space, it suffices to prove that Assumption 1 holds true for $L^{p(\cdot)}(\rn)$. By \cite[Proposition 2.37]{caf13}, we obtain that the associate space of $L^{p(\cdot)}(\rn)$ is
 $L^{p'(\cdot)}(\rn)$, where $1/p(\cdot)+1/p'(\cdot)=1$. By \cite[Theorem 1.2]{c06}, we obtain that $p'(\cdot) \in \mathscr{B}(\rn)$.  Thus, the Hardy-Littlewood maximal operator $M_{\mathrm{HL}}$ is bounded on $L^{p(\cdot)}(\rn)$  and $L^{p'(\cdot)}(\rn)$. This completes the proof of Theorem \ref{ta4.1x}.
\end{proof}

\subsection{Mixed Herz-Hardy spaces}\label{s4.2}
In this subsection, we show that the mixed-norm Lebesgue space is a ball Banach function space. Therefore, mixed Herz-Hardy spaces (see \cite{z22}) belong to the family of Herz-type Hardy spaces associated with ball quasi-Banach function spaces. Moreover, the atomic decompositions of $\mathcal{H}\dot{\mathcal{K}}_{X}^{\alpha,\,q}(\mathbb{R}^n)$ (see Theorem \ref{t3.1}) and the boundedness of certain sublinear operators from $\mathcal{H}\dot{\mathcal{K}}_{X}^{\alpha,\,q}(\mathbb{R}^n)$ to $\dot{\mathcal{K}}_{X}^{\alpha,\,q}(\mathbb{R}^n)$ (see Theorem \ref{t5.1}) can be applied to the aforementioned mixed Herz-Hardy spaces, i.e., the following Theorem \ref{ta4.4x}. To prove this, we need to recall several definitions.

The theory of mixed-norm function spaces can be traced back
to the work of Benedek and Panzone \cite{b61}.
 Later on, in 1970, Lizorkin \cite{l70}
further developed both the theory of multipliers of Fourier integrals and estimates of
convolutions in the mixed-norm Lebesgue spaces.  In recent years, inspired by \cite{b61} and \cite{h60}, the theory of mixed-norm
function spaces,  including mixed-norm Morrey spaces, mixed-norm Hardy spaces,
mixed-norm Besov spaces and mixed-norm Triebel-Lizorkin spaces, has rapidly been developed(see, for instance \cite{c17,g17,h19,n19}). Now, we present the definition of mixed-norm Lebesgue spaces from \cite{b61} as
follows.
\begin{definition}
 Let $\vec{p} := (p_1,\cdot\cdot\cdot, p_n) \in(0,\infty)^n$. The mixed-norm Lebesgue space
$L^{\vec{p}}(\rn)$ is defined to be the set of all measurable functions $f$ such that
$$\|f\|_{L^{\vec{p}}(\rn)}:=\lf\{\int_{\rr}\cdot\cdot\cdot\lf[\int_{\rr}
|f(x_1,\cdot\cdot\cdot,x_n)|^{p_1}dx_{1}\r]^{p_2/p_1}\cdot\cdot\cdot dx_{n}\r\}^{1/p_n}<\infty$$
with the usual modifications made when $p_i=\infty$ for some $i\in \{1,\cdot\cdot\cdot, n\}$.
\end{definition}

From \cite[p.9]{s17} and \cite[Section. 5]{wy19}, we know that $L^{\vec{p}}(\rn)$ with $\vec{p} := (p_1,\cdot\cdot\cdot, p_n) \in(0,\infty)^n$ is a ball
quasi-Banach function space. Furthermore, it is easy to see that $L^{\vec{p}}(\rn)$ is a ball Banach function space if $\vec{p} = (p_1,\cdot\cdot\cdot, p_n) \in(1,\infty)^n$.
Let $X=L^{\vec{p}}(\rn)$ with $\vec{p} = (p_1,\cdot\cdot\cdot, p_n) \in(1,\infty)^n$.  For any given $q\in(0,\,\infty)$ and  $\alpha\in \mathbb{R}$, $\mathcal{H}\dot{\mathcal{K}}_{X}^{\alpha,\,q}(\rn)$  is the mixed Herz-Hardy space $\mathcal{H}\dot{\mathcal{K}}_{L^{\vec{p}}}^{\alpha,\,q}(\rn)$  (see \cite[Definition\,2.4]{z22}), respectively.  By applying Theorems \ref{t3.1} and  \ref{t5.1} to $X=L^{\vec{p}}(\rn)$, we obtain the following theorem.

\begin{theorem}\label{ta4.4x}
 For any  $\vec{p} = (p_1,\cdot\cdot\cdot, p_n) \in(1,\infty)^n$, Theorems \ref{t3.1} and \ref{t5.1} hold true with $X$ replaced by $L^{\vec{p}}(\rn)$.
\end{theorem}

\begin{remark}
  Let $\vec{p} = (p_1,\cdot\cdot\cdot, p_n) \in(1,\infty)^n$.  When $X:=L^{\vec{p}}(\rn)$,  for any  $q\in(0,\,\infty)$ and $\alpha\in \mathbb{R}$, we have  $\mathcal{H}\dot{\mathcal{K}}_{X}^{\alpha,\,q}(\rn)=\mathcal{H}\dot{\mathcal{K}}_{L^{\vec{p}}}^{\alpha,\,q}(\rn)$. In this case,  Theorems \ref{t3.1} and \ref{t5.1} coincide with  \cite[ Theorems 3.1 and 5.1]{z22}.
\end{remark}

\begin{proof}[Proof of Theorem \ref{ta4.4x}]
Let $\vec{p} = (p_1,\cdot\cdot\cdot, p_n) \in(1,\infty)^n$. Since the space $L^{\vec{p}}(\rn)$ is a ball Banach function space, it suffices to prove that Assumption 1 holds true for $L^{\vec{p}}(\rn)$. By \cite[Theorem\,1]{b61}, we obtain that the associate space of $L^{\vec{p}}(\rn)$ is
 $L^{{\vec{p}}\,'}(\rn)$, where $\vec{p\,}^{\prime}:=(p_1^{\prime},\ldots, p_n^{\prime})$ with $1/p_i+1/p_i^{\prime}=1$ for any $i\in\{1,\ldots,n\}$.  Thus, by \cite[Lemma\,3.5]{h19}, we obtain that the Hardy-Littlewood maximal operator $M_{\mathrm{HL}}$ is bounded on $L^{\vec{p}}(\rn)$  and $L^{{\vec{p}}\,'}(\rn)$. This completes the proof of Theorem \ref{ta4.4x}.
\end{proof}

\subsection{ Herz-Orlicz Hardy spaces}\label{s4.3}

In this subsection, we introduce a new Herz-type Hardy space: {\it{Herz-Orlicz Hardy space}},  which belongs to the family of Herz-type Hardy spaces associated with ball quasi-Banach function spaces. We present  that the Orlicz function space is also a ball Banach function space. Consequently,  the atomic decompositions of $\mathcal{H}\dot{\mathcal{K}}_{X}^{\alpha,\,q}(\mathbb{R}^n)$ (see Theorem \ref{t3.1}) and the boundedness of certain sublinear operators from $\mathcal{H}\dot{\mathcal{K}}_{X}^{\alpha,\,q}(\mathbb{R}^n)$ to $\dot{\mathcal{K}}_{X}^{\alpha,\,q}(\mathbb{R}^n)$ (see Theorem \ref{t5.1}) can be applied to the aforementioned  Herz-Orlicz Hardy spaces , i.e., the following Theorem \ref{t4.9ax}. To prove this, we need to recall several key definitions and lemmas.

 We first recall Orlicz spaces. Note that there are many operators that are not bounded on Lebesgue spaces \(L^{p}(\mathbb{R}^{n})\)  with \(p \in[1, \infty]\) especially on \(L^{1}(\mathbb{R}^{n})\)  and \(L^{\infty}(\mathbb{R}^{n})\). Orlicz spaces are used to cover the failure of the boundedness of some integral operators. For example, the Hardy-Littlewood maximal operator $M_{\mathrm{HL}}$ fails to be bounded on \(L^{1}(\mathbb{R}^{n})\) but bounded from \(L^{1}(\mathbb{R}^{n})\) to \(L^{1}(log L)^{1+\varepsilon}(\mathbb{R}^{n})\)  with \(\varepsilon \in(0, \infty)\), where \(L^{1}(log L)^{1+\varepsilon}(\mathbb{R}^{n})\)  with \(\varepsilon \in(0, \infty)\) is a special Orlicz space. Consequently,  the Orlicz spaces have attracted significant attention, owing to their  finer and subtler structures.  For more study of Orlicz spaces, we
refer the reader to \cite{ha19,l11,r91}.

 Recall that a function \(\Phi:[0, \infty) \to [0, \infty)\) is called an {\it Orlicz function} if it is non-decreasing and satisfies \(\Phi(0)=0\) , \(\Phi(t)>0\) whenever \(t \in(0, \infty)\), and \(\lim _{t \to \infty} \Phi(t)=\infty.\) An Orlicz function $\Phi$ is said to be of  positive lower (resp, upper) type $p$ with \(p \in(0, \infty)\) if there exists a positive constant \(C_{p}\), depending on $p$, such that, for any \(t \in[0, \infty)\) and \(s \in(0,1)\) (resp., \(s \in[1, \infty))\),
\[\Phi(s t) \leq C_{p} s^{p} \Phi(t).\]
A convex left-continuous Orlicz function is called a {\it Young function}.  The following definition of  Orlicz space is from \cite{r91}.
\begin{definition}
 Let $\Phi$ be an Orlicz function with positive lower type \(p_{\Phi}^{-}\) and positive upper type \(p_{\Phi}^{+}\). {\it The Orlicz space} \(L^{\Phi}(\mathbb{R}^{n})\) is defined to be the set of all measurable functions $f$ on \(\mathbb{R}^{n}\) such that
\[\| f\| _{L^{\Phi}\left(\mathbb{R}^{n}\right)}:=\inf \left\{\lambda \in(0, \infty): \int_{\mathbb{R}^{n}} \Phi\left(\frac{|f(x)|}{\lambda}\right) d x \leq 1\right\}<\infty .\]
\end{definition}
As demonstrated in  \cite[Section 7.6]{s17},  when $\Phi$ is an Orlicz function with positive lower type \(p_{\Phi}^{-}\in[1,\,\infty)\) and positive upper type \(p_{\Phi}^{+}\),  the space \(L^{\Phi}(\mathbb{R}^{n})\) is a ball Banach function space. Now, let $\Phi$ be a Young function satisfying the conditions of positive lower type \(p_{\Phi}^{-}\in[1,\,\infty)\) and positive upper type \(p_{\Phi}^{+}\). By setting $X=L^{\Phi}(\mathbb{R}^{n})$ in Definition \ref{da2.7},  a novel {\it Herz-Orlicz Hardy space} $\mathcal{H}\dot{\mathcal{K}}_{L^{\Phi}}^{\alpha,\,q}(\rn)$ is thereby established.  Applying Theorems \ref{t3.1} and  \ref{t5.1} to $X=L^{\Phi}(\mathbb{R}^{n})$, we obtain the atomic decompositions of $\mathcal{H}\dot{\mathcal{K}}_{L^{\Phi}}^{\alpha,\,q}(\rn)$  and the boundedness of certain sublinear operators from $\mathcal{H}\dot{\mathcal{K}}_{L^{\Phi}}^{\alpha,\,q}(\rn)$ to $\dot{\mathcal{K}}_{L^{\Phi}}^{\alpha,\,q}(\mathbb{R}^n)$.
\begin{theorem}\label{t4.9ax}
Let $\Phi$ be a Young function with positive lower type \(p_{\Phi}^{-}\in(1,\,\infty)\) and positive upper type \(p_{\Phi}^{+}<\infty\). Then, Theorems \ref{t3.1} and \ref{t5.1} hold true with $X$ replaced by \(L^{\Phi}(\mathbb{R}^{n})\).
\end{theorem}
To prove Theorem \ref{t4.9ax}, we need
following the definition of complementary function to introduce the associate spaces of the Orlicz spaces.
\begin{definition}
 Let $\Phi$ be a Young function. {\it The complementary function} of $\Phi$ is defined as
\[\Phi^{*}(t)=\sup \{s t-\Phi(s): s \in[0, \infty)\}, \ \ \ t \geq 0 .\]
\end{definition}
The complementary function of $\Phi$ is also called as the conjugate function of $\Phi$  (see \cite[Section 2.4]{ha19}). If $\Phi$ is a Young function, then the complementary function of \(\Phi^{*}\) is $\Phi$ (detailed in \cite[Proposition 2.4.5]{ha19}).
\begin{lemma}{\rm\cite[Proposition 2.4.9]{ha19}}\label{axl4.8}
Let  $\Phi$ be a Young function and p, \(q \in(1, \infty)\). Then
 \begin{enumerate}
\item[\rm{(i)}]
 $\Phi$ is of lower type p if and only if \(\Phi^{*}\) is of upper type \(p'\);
 \item[\rm{(ii)}]
 $\Phi$ is of upper type q if and only if \(\Phi^{*}\) is of lower type \(q'\).
 \end{enumerate}
\end{lemma}

\begin{proof}[Proof of Theorem \ref{t4.9ax}]
Let $\Phi$ be an Orlicz function with positive lower type \(p_{\Phi}^{-}\in(1,\,\infty)\) and positive upper type \(p_{\Phi}^{+}<\infty\). Since the space \(L^{\Phi}(\mathbb{R}^{n})\) is a ball Banach function space, it suffices to prove that Assumption 1 holds true for $L^{\Phi}(\mathbb{R}^{n})$. By \cite[Theorem\,3.4.6]{ha19}, we obtain that the associate space of ${L^{\Phi}}(\rn)$ is
 ${L^{\Phi^{*}}}(\rn)$. From Lemma \ref{axl4.8}, it follows that when $\Phi$ is a Young function with positive lower type \(p_{\Phi}^{-}\in(1,\,\infty)\) and positive upper type \(p_{\Phi}^{+}<\infty\),
 \(\Phi^{*}\) is also a Young function with positive lower type \(\lf(p_{\Phi}^{+}\r)'\in(1,\,\infty)\) and positive upper type \(\lf(p_{\Phi}^{-}\r)'<\infty\). Thus, by \cite[Corollary 2.1.2]{y17}, we obtain that the Hardy-Littlewood maximal operator $M_{\mathrm{HL}}$ is bounded on $L^{\Phi}(\mathbb{R}^{n})$ and ${L^{\Phi^{*}}}(\rn)$. This completes the proof of Theorem \ref{t4.9ax}.
\end{proof}


\bigskip

\noindent Aiting Wang and Baode Li \\
\noindent College of Mathematics and System Sciences, Xinjiang University\\
Urumqi,
830017, Xinjiang, China\\
\noindent{E-mail }:\\
atwangmath@163.com (Aiting Wang)\\
baodeli@xju.edu.cn (Baode Li)
\medskip

\noindent Wenhua Wang\\
\noindent
College of Mathematics and Statistics, Wuhan University\\
Luo-Jia-Shan, Wuhan
430072, Hubei, China\\
\noindent{E-mail }:whwangmath@whu.edu.cn

\medskip
\noindent Mingquan Wei\\
\noindent
Xinyang Normal University\\
Xinyang, 464000, Henan, China\\
\noindent{E-mail }:weimingquan11@mails.ucas.ac.cn

\bigskip \medskip

\end{document}